\documentclass[11pt,reqno,palentino]{amsart} 

\usepackage{mypack}

\usepackage{latexsym}
\usepackage{a4wide}
\usepackage{amscd}
\usepackage{graphics}
\usepackage{amsmath}
\usepackage{amssymb}
\usepackage{mathrsfs}
\usepackage{accents}
\usepackage{enumerate}
\usepackage{soul,xcolor}
\usepackage{hyperref}
\usepackage{soul}%striketrough
\usepackage[cal=boondoxo]{mathalfa}

\DeclareMathOperator{\id}{id}

\DeclareMathOperator{\sdim}{sdim}
\DeclareMathOperator{\im}{im}

\newcommand{\cP}{\ensuremath{\mathcal P}}

\newcommand{\cO}{\ensuremath{\mathcal O}}

\newcommand{\cS}{\ensuremath{\mathcal{S}}}
\newcommand{\cF}{\ensuremath{\mathcal{F}}}
\newcommand{\cf}{\ensuremath{\mathcal{f}}}
\newcommand{\cg}{\ensuremath{\mathcal{g}}}
\newcommand{\ch}{\ensuremath{\mathcal{h}}}
\newcommand{\cid}{\ensuremath{\mathcal{id}}}
\author[F. Pasquotto]{Federica Pasquotto}
\author[T.O. Rot]{Thomas O. Rot}
\thanks{T.O. Rot is supported by NWO-NWA startimpuls - 400.17.608.}

\title{Degree theory for orbifolds}
\begin{document}

\maketitle

\begin{abstract}
In \cite{BB2012} Borzellino and Brunsden started to develop an elementary differential topology theory for orbifolds. In this paper we carry on their project by defining a mapping degree for proper maps between orbifolds, which counts preimages of regular values with appropriate weights. We show that the mapping degree satisfies the expected invariance properties, under the assumption that the domain does not have a codimension one singular stratum. We study properties of the mapping degree and compute the degree in some examples.
\end{abstract}

\section{Introduction}
An orbifold is a space that is locally homeomorphic to $\mR^n/G$, where a finite group $G$ acts linearly on $\mR^n$. It is therefore natural to expect that many properties of manifolds are shared by orbifolds, such as the existence of an elementary differential topology. In this paper we contribute to the project of Borzellino and Brunsden~\cite{BB2012} in this direction and we define a mapping degree for complete orbifold maps. 

More precisely, for a proper, smooth orbifold map $\cf:\cO \rightarrow \cP$ and a regular value $y\in P$ of $\cf$, we define the mod-$2$ degree of $\cf$ at $y$, $\deg_2(\cf;y)$ and, if $\cO$ and $\cP$ are oriented, the integer valued degree $\deg(\cf;y)$.
These degrees are defined by a weighted count of the preimages of a given regular value: this takes into account the possibly non-trivial isotropy of the points involved. 

Our first result in Section~\ref{sec:degree} is that the mod-$2$ degree $\deg_2(\cf;y)$ is independent of the choice of regular value $y\in \cP$ and of the proper homotopy class of $\cf$, provided $\cP$ is connected and $\cO$ does not contain a codimension-$1$ singular stratum. We show that the condition on the codimension-$1$ stratum is necessary by constructing  in Section~\ref{sec:examples} an orbifold map whose degree depends on the regular value chosen. If the orbifolds $\cP$ and $\cO$ are oriented, this condition is always satisfied, since orientable orbifolds never have codimension-$1$ singular points. It follows that in the oriented case, if $\cP$ is connected, then the integer degree $\deg(\cf;y)$ is well defined and always independent of both the choice of regular value $y\in \cP$ and the proper homotopy class of $\cf$. 

In Section~\ref{sec:degree} we also prove some properties of the degree, namely that non-zero degree implies that the underlying map of an orbifold map is surjective. We also show that the degree is multiplicative under compositions of maps. 

Orbifolds frequently arise as quotients $M//G$, where a compact Lie group $G$ acts effectively on a manifold $M$ with finite stabilizers. These quotients are also a natural source of orbifold mappings. We show in Section~\ref{sec:quotient} that an equivariant map between such manifolds with group actions induces a complete orbifold map. We use this to construct orbifold mappings between weighted projective spaces, and we calculate their degree in Section~\ref{sec:examples}.

\section{Orbifolds and orbifold maps}
\label{sec:definitions}
For us an \emph{(effective) orbifold} $\mathcal{O}$ is a topological space $O$ together with an atlas of orbifold charts around each point of $O$. An \emph{orbifold chart} around a point $x\in O$ consists of an open set $\tilde{U}_x\subset \mathbb{R}^n$ together with an effective action of \emph{the isotropy group}, a finite group $\Gamma_x$, fixing $0\in \mathbb{R}^n$, and a homeomorphism $\phi_x$ from $\tilde{U}_x/\Gamma_x$ onto an open neighborhood of $U_x$ of $x$ such that $\phi_x(0)=x$.
Two charts $\tilde{U}_x$ and $\tilde{U}_y$ with $U_x\cap U_y \not=\emptyset$ need to satisfy the following compatibility condition in the atlas: there exist an orbifold chart $\tilde U_z$, injective group homomorphisms $\rho_{zx}:\Gamma_z\rightarrow \Gamma_x$ and $\rho_{zy}:\Gamma_z\rightarrow \Gamma_y$, and embeddings $i_{zx}: \tilde{U}_z\rightarrow \tilde{U}_x$ and $i_{zy}: \tilde{U}_z\rightarrow \tilde{U}_y$, such that 
$i_{zx}(\gamma\tilde{w})=\rho_{zx}(\gamma)(i_{zx}(\tilde{w}))$ and $i_{zy}(\gamma\tilde{w})=\rho_{zy}(\gamma)(i_{zy}(\tilde{w}))$ for all $\gamma\in \Gamma_z$ and $\tilde{w}\in\tilde{U}_z$. 

Just as in the manifold case, every orbifold atlas lies in a unique maximal atlas and we will always assume our orbifolds to be equipped with a maximal atlas. The Bochner-Cartan linearization theorem shows that we may always choose charts in which the groups $\Gamma_x$ act linearly on $\tilde U_x=\mR^n$, i.e., the chart is a representation of $\Gamma_x$. 

The \emph{singular set} $\Sigma$ of $\mathcal{O}$ consists of the points with a non-trivial isotropy group. 
Let $x\in \Sigma$ and $\tilde x$ the lift of $x$ to a chart $\tilde U_x$ centered at $x$. Then $\Gamma_x$ acts on $T_{\tilde x} \tilde U_x$ and fixes a subspace $T_{\tilde x}\tilde U_x^{\Gamma_x}$, cf. \cite{D2015}. The \emph{singular dimension} of $x$ is defined to be $\sdim(x)=\dim T_{\tilde x}\tilde U_x^{\Gamma_x}$ and does not depend on the chosen chart. The singular set is the union $\Sigma=\bigcup_{i= 0}^{n-1} \Sigma_i$ of \emph{singular strata} $\Sigma_i$, where $\Sigma_i=\{x\in \Sigma\,\vert\,\sdim(x)=i\}$. Each stratum $\Sigma_i$ can be further decomposed: the connected components of $\Sigma_i$ all have a well defined isotropy group up to isomorphism. We will not use this decomposition. 

  It can be shown that each stratum $\Sigma_i$ is a boundary-less manifold of dimension $i$, cf. \cite[Page 74 and onwards]{DThesis}, whose tangent spaces are modeled on $T_{\tilde x} \tilde U_x^{\Gamma_x}$. We will refer to points in $\Sigma_{\dim \cO}$ as \emph{smooth} points, and points in $\Sigma_i$ with $i<\dim \cO$ as \emph{non-smooth points}.
  
We will also consider smooth orbifolds with boundary, by allowing the orbifold charts $\tilde{U}_x$ in the definition of orbifold to be open subsets of the closed upper half-space  $[0,\infty)\times \mathbb{R}^{n-1}$. 

An orbifold is \emph{locally orientable} if the groups $\Gamma_x$ act on $\tilde U_x$ by orientation preserving diffeomorphisms. An orbifold is \emph{orientable} if in addition the embeddings $i_{xy}:\tilde U_x\rightarrow \tilde U_y$ are orientation preserving. An \emph{orientation} is then a consistent choice of orientation of the charts $\tilde U_x$. We will need the following result on the structure of the singular set.
  
  \begin{proposition}
    \label{prop:structure}
    Let $\cO$ be an orbifold. Then $\Sigma_{\dim \cO}$ is an open and dense subset which is a manifold. The union $\Sigma_{\dim \cO}\cup \Sigma_{\dim \cO-1}\subset O$ is a manifold with boundary $\partial(\Sigma_{\dim \cO}\cup \Sigma_{\dim \cO-1})=\Sigma_{\dim \cO-1}$. If $\cO$ is connected, then $\Sigma_{\dim \cO}$ is connected. If $\cO$ is orientable, then $\Sigma_{\dim \cO-1}=\emptyset$.
  \end{proposition}
  \begin{proof}
    Clearly $\Sigma_{\dim \cO}$ is a manifold and it is open in $O$. Moreover, since we demand that the isotropy groups $\Gamma_x$ act effectively, it is also dense. Now suppose that $\Sigma_{\dim \cO-1}\not=\emptyset$. Let $x\in \Sigma_{\dim \cO-1}$. Let $\tilde U_x$ be an orbifold chart centered at $x$ on which the finite group $\Gamma_x$ acts. Without loss of generality, we can assume $\tilde U_x\cong \mR^n$ and that the action is linear. The action then fixes a hyperplane $V\subset \mR^n$. Choose a $\Gamma_x$ invariant inner product. Then $\Gamma_x$ also leaves the line $V^\perp=\{v\in \mR^n\,\vert\,\langle v,w\rangle=0,\text{ for all } w\in V\}$ invariant and it must act effectively on $V^\perp$. But this is only possible if $\Gamma_x=\mZ_2$ and the action is by a reflection. The quotient of $\mR^n$ by this action is thus isomorphic to $(V^\perp/{\mathbb{Z}_2})\times V\cong [0,\infty)\times \mR^{n-1}$. We see that $\Sigma_{\dim \cO}\cup \Sigma_{\dim \cO-1}$ is a manifold with boundary $\Sigma_{\dim \cO-1}$. Notice that $\Sigma_{\dim\cO} \cup \Sigma_{\dim \cO-1}$ is obtained from $\cO$ by removing sets of codimension $2$ and higher. Thus $\Sigma_{\dim\cO} \cup \Sigma_{\dim \cO-1}$ is connected if $\cO$ is connected. A connected manifold with boundary has a connected interior, hence $\Sigma_{\dim \cO}$ is connected. A reflection is not orientation preserving, so if $\cO$ is orientable, then $\Sigma_{\dim \cO-1}=\emptyset$. 
  \end{proof}

A \emph{full suborbifold} $\cS$ of the orbifold $\cO$ consists of:
  \begin{itemize}
  \item a subspace $S\subset O$; 
  \item for each $x\in S$ and neighborhood $W$ of $x$ in $O$, a linear orbifold chart $(\tilde U_x,\Gamma_x,\phi_x)$, with $U_x\subset W$, and a $\Gamma_x$-invariant linear subspace $\tilde V_x\subset \tilde U_x$, such that $(\tilde V_x,\Gamma_x/\Omega_x,\phi_x\bigr\vert_{\tilde V_x})$ is an orbifold chart for $\cS$ (where $\Omega_x=\{\gamma\,:\, \gamma\bigr\vert_{\tilde V_x}=\id\}$) and $V_x:=\phi_x(\tilde V_x/\Gamma_x)=U_x\cap S.$
\end{itemize}
    Of course all maps in the above definition are required to be smooth. We refer to \cite{BB2012,BB2015} for a general definition of suborbifold and a discussion of the properties of full suborbifolds.

Given two orbifolds $\mathcal{O}$ and $\mathcal{P}$, a \emph{smooth complete orbifold map} between $\mathcal{O}$ and $\mathcal{P}$ is a triple $\cf=(f,\{\tilde f_x\},\{\Theta_{x}\})$ consisting of the following data:
\begin{itemize}
\item a continuous map $f:O\rightarrow P$ between the underlying topological spaces; 
\item for each $x\in O$, a group homomorphism $\Theta_{x}:\Gamma_x\rightarrow \Gamma_{f(x)}$; 
\item for each $x\in O$, given orbifold charts $(\tilde{U}_x,\Gamma_x)$ and $(\tilde{U}_{f(x)},\Gamma_{f(x)})$ around $x$ and $f(x)$, respectively, a smooth lift $\tilde{f}_x:\tilde{U}_x\rightarrow \tilde{U}_{f(x)}$ which is $\Theta_{x}$-equivariant, i.e. $\tilde f_x(\gamma y)=\Theta_x(\gamma)\tilde f_x(y)$ for all $\gamma\in \Gamma_x$.
  
\end{itemize}
We identify two orbifold maps $\cf=(f,\{\tilde f_x\},\{\Theta^f_x\})$ and $\cg=(g,\{\tilde g_x\},\{\Theta^g_x\})$ if for every $x\in O$ there exists an orbifold chart $(\tilde{U}_x,\Gamma_x)$ such that $\tilde{f}_x\vert_{\tilde{U}_x}=\tilde{g}_x\vert_{\tilde{U}_x}$ and $\Theta^f_x=\Theta^g_{x}$. 

In particular, the maps $f$ and $g$ of the underlying topological spaces coincide for equivalent maps. We will often drop the adjective complete, and speak of a smooth (orbifold) map $\cf$.

A smooth orbifold map $\cf:\cO\rightarrow \cP$ is an \emph{orbifold diffeomorphism} if there exists a smooth orbifold map $\cf^{-1}:\cP\rightarrow \cO$ such that $\cf^{-1}\cf=\cid_\cO$ and $\cf \cf^{-1}=\cid_\cP.$ The identity $\cid$ is the smooth orbifold map defined by a triple where all  maps involved are identities.
A smooth orbifold map $\cf:\cO\rightarrow \cP$ is \emph{proper} if the underlying map $f:O\rightarrow P$ is proper: i.e. preimages of compact sets are compact. If $\cO$ is compact, then $\cf$ is automatically proper. 

In order to define homotopies of orbifold maps, we need to consider the product orbifold structure on $\mathcal{O}\times [0,1]$ (\cite{BB2008}): its singular set consists of points of the form $(x,t)$, with $x\in \Sigma$ and $t\in [0,1]$, and for all such points the isotropy group $\Gamma_{(x,t)}$ is isomorphic to $\Gamma_x$. The time $t$ inclusion $\mathcal{i}_t:\cO\rightarrow \cO\times [0,1]$ is a smooth orbifold map, where all the defining data are the obvious inclusions.  

Two smooth orbifold maps $\cf,\,\cg: \mathcal{O}\rightarrow \mathcal{P}$ are called \emph{smoothly homotopic} if there exists a smooth orbifold map $\cF:\, \mathcal{O}\times [0,1]\rightarrow \mathcal{P}$ such that 
\[
\cF\mathcal{i}_0=\cf\quad \textrm{and}\quad \cF\mathcal{i}_1=\cg. 
\]
The homotopy is said to be proper if the underlying map $F$ is proper (this assumption is stronger than the assumption that $\cF\mathcal i_t$ is proper for all $t$, cf.~\cite{rot}).

The statements below are the orbifold version of the well-known Sard's theorem and regular preimage theorem for manifolds. 
Given a smooth orbifold map $\cf:\cO\rightarrow \cP$, a point $x\in O$ is called \emph{regular} if the differential
\[
T_{\tilde x}\tilde{f}_x:T_{\tilde{x}}\tilde{U}_x\rightarrow T_{\tilde{f}(\tilde{x})}\tilde{U}_y
\]
is a surjective linear map.  Here $\tilde{x}$ denotes the lift of $x$ to $\tilde{U}_x$. A point $y\in P$ is called a \emph{regular value} if all $x\in f^{-1}(y)$ are \emph{regular points}. 
We refer the reader to \cite{BB2012} for proofs of the statements.

\begin{theorem}~\cite[Theorem 4.1]{BB2012}
Let $\cf:\cO\rightarrow \cP$ be a complete orbifold map. Then the set of regular values is is dense in $\cP$.
\end{theorem}

\begin{remark}
Since a proper map is closed and the set of critical points is closed, the set of regular values of a proper map is open and dense. 
\end{remark}

\begin{theorem}~\cite[Theorems 4.2 and 4.3]{BB2012}
  \label{thm:preimage}
Let $\cO$ and $\cP$ be orbifolds. Let $\cf:\cO\rightarrow \cP$ be a smooth orbifold map and $y\in P$ a regular value. Then $\cf^{-1}(y)$ has the structure of a full suborbifold of dimension $\dim \cO-\dim \cP$. If $\cO$ has boundary and $y$ is also regular for $\cf\bigr\vert_{\partial \cO}$, then $\cf^{-1}(y)$ is a full suborbifold with boundary contained in $\partial \cO$. 
\end{theorem}

  It was remarked in~\cite{BB2012} that there are representation theoretic obstructions for a point of a smooth orbifold map $\cf:\cO\rightarrow \cP$ to be regular. We need the following version of this principle. 

  \begin{proposition}
    \label{prop:singularity}
Let $\cf:\cO\rightarrow \cP$ be a smooth orbifold map. Let $y\in P$ be a regular value and a smooth point of $\cP$. Then $\cf^{-1}(y)$ is a full suborbifold and every $x\in f^{-1}(y)$ has singular dimension $\sdim_{\cO}(x)\geq \dim \cP$.
\end{proposition}
\begin{proof}
  Let $\tilde U_x$ be a chart centered at $x$ and $\tilde x$ be the lift of $x$. Then $\Gamma_x$ acts on $T_{\tilde x} U_x$ by the differential. To avoid cumbersome notation we denote this action by left multiplication, i.e $\gamma \cdot \tilde v=T\gamma_{\tilde x} \tilde v$.  Following~\cite{BB2012} we define $K_x=\ker T_{\tilde x}\tilde f_x$ and $N_x=\ker \Theta_x$. Now define the linear operator $A_x$ by
  $$
  A_x \tilde v=\tilde v-\frac{1}{\vert N_x \vert}\sum_{\gamma \in N_x}\gamma\cdot \tilde v.
  $$
Here $|N_x|$ denotes the order of $N_x$. The operator has the following properties. If $\tilde v\in T_{\tilde x}\tilde{U}_x$ then $T_{\tilde{x}}\tilde f_x(\gamma\cdot \tilde v)=T_{\tilde{x}}\tilde f_x (\tilde v)$ for all $\gamma\in N_x$ hence $T_{\tilde{x}}\tilde f_x (A_x \tilde v)=0$, i.e. $\im A_x\subset K_x$. Moreover, the operator $A_x$ commutes with the action and $A_x^2=A_x$. Thus $A_x:T_{\tilde x}\tilde U_x\rightarrow K_x$ is an $N_x$ invariant linear projection (which might not be surjective).

  If $\tilde v\in \ker A_x$ then $\tilde v=\frac{1}{\vert N_x\vert}\sum_{\gamma\in N_x} \gamma\cdot \tilde v$ hence
\[
\delta \tilde v=\frac{1}{\vert N_x\vert}\sum_{\gamma}\delta\cdot (\gamma\cdot v)=\frac{1}{\vert N_x\vert}\sum_{\gamma}(\delta \gamma)\cdot v=\frac{1}{\vert N_x\vert}\sum_{\gamma'\in N_x}\gamma'\cdot v=\tilde v
\]
for all $\delta\in N_x$. So far we have not used the assumption that $y$ is a smooth point. But if that is the case, then $N_x$ must be the full isotropy group, i.e. $N_x=\Gamma_{x}^\cO$. It follows that $\sdim_\cO(x)\geq \dim \ker A_x$, and therefore $\sdim_\cO (x) \geq \dim \ker A_x=\dim \cO-\dim \im A_x\geq \dim \cO-\dim K_x$ by the rank-nullity Theorem. Since $y$ is regular, we have that $\dim K_x=\dim \cO-\dim \cP$, which implies that $\sdim_\cO(x)\geq \dim \cP$. 
\end{proof}

  \begin{corollary}
    Let $\cf:\cO\rightarrow \cP$ be a smooth orbifold map between orbifolds of the same dimension. Let $y\in P$ be a regular value and a smooth point of $\cP$. Then each $x\in \cf^{-1}(y)$ is smooth.
\end{corollary}

\section{Degree of an orbifold map}
\label{sec:degree}
\begin{figure}
\def\svgwidth{.5\textwidth}
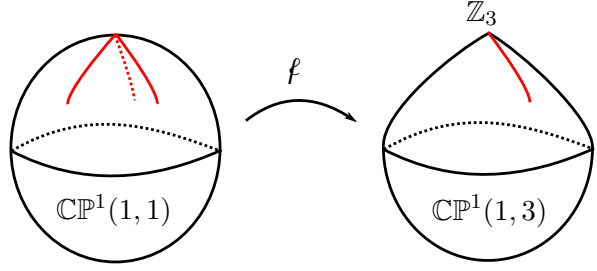

\caption{In Section~\ref{sec:examples} we define complete orbifold maps between weighted projective spaces. This picture shows the map $\cf_{(1,3)}:\mathbb{CP}^1(1,1)\rightarrow \mathbb{CP}^1(1,3)$ with underlying map $f\left([z_0,z_1]_{(1,1)}\right)=[z_0,z_1^3]_{(1,3)}$. The orbifold $\mathbb{CP}^1(1,3)$ has one non-smooth point $[0,1]_{(1,3)}$ with isotropy $\mathbb{Z}_3$, depicted at the top on the right hand side. Even though the point is non-smooth, it is a regular value of $\cf$. The red curve on the right therefore consists of regular values. The preimage of every point on this curve consists of three points, except for the non-smooth point. Thus the cardinality of the preimage is not locally constant. The weighted cardinality however, is, as we show in~Lemma~\ref{lemma:localinvariance}. 
}
  \label{fig:regularvalue}
\end{figure}

In the case of smooth maps between manifolds of the same dimension, if $y$ is a regular value of a proper map, then the cardinality $\#f^{-1}(y)$ is a locally constant function. For an orbifold map, this is not going to be the case, unless we assign \emph{weights} to the points in the preimage, cf.~Figure~\ref{fig:regularvalue}.

  \begin{definition}
Let $\cO$ and $\cP$ be orbifolds of the same dimension. The \emph{weighted cardinality} of the preimage of a regular point $y\in P$ of a smooth proper orbifold map $\cf:\cO\rightarrow \cP$ is the number
\[
\#_w \cf^{-1}(y)=\sum_{x\in f^{-1}(y)} \frac{\vert\Gamma_y\vert}{\vert\Gamma_x\vert}.
\]
\end{definition}
The next Lemma tells us that the weighted cardinality is, in fact, an integer number.
\begin{lemma}
  Let $\cf:\cO\rightarrow \cP$ be a proper and smooth orbifold map between orbifolds of the same dimension. Let $x$ be a regular point and $y=f(x)$. 
 Then $\frac{\vert\Gamma_{y}\vert}{\vert\Gamma_x\vert}$ is an integer.
\end{lemma}
\begin{proof}
  To avoid possible confusion, in this proof we label the isotropy groups by the orbifold to which they belong. 
  Part of the data defining a smooth orbifold map is a group homomorphism $\Theta_x:\Gamma^{\cO}_x\rightarrow \Gamma^{\cP}_y$. By ~\cite[Theorem 5.3]{BB2012} the group $N_x=\ker \Theta_x$ can be viewed as a subgroup of $\Gamma^{\cS}_x$, where $\mathcal S=\cf^{-1}(y)$. But $\mathcal S$ is a zero dimensional orbifold, hence a manifold. This implies that $\Gamma^{\cS}_{x}$ is the trivial group and hence that $\Theta_x$ is injective whenever $x$ is a regular point. Then $\vert\im \Theta_x\vert=\frac{\vert\Gamma^{\cO}_x\vert}{\vert\ker \Theta_x\vert}=\vert\Gamma^{\cO}_x\vert.$ The index $[\Gamma^{\cP}_y:\im \Theta_x]=\frac{\vert\Gamma^{\cP}_y\vert}{\vert\im \Theta_x\vert}=\frac{\vert\Gamma^{\cP}_y\vert}{\vert\Gamma^{\cO}_x\vert}$ of a subgroup is always an integer, which was what was to be shown.
\end{proof}

\begin{lemma}
  \label{lemma:localinvariance}
Suppose $\cf:\cO\rightarrow \cP$ is a smooth proper orbifold map between orbifolds of the same dimension and $y$ is a regular value of $\cf$. Then the number $\#_w\cf^{-1}(y)$ is locally constant.
\end{lemma}
\begin{proof}
Choose orbifolds charts $U_y$ and $U_x$ around around $y$ and $x\in f^{-1}(y)$, respectively. After shrinking $U_y$ and $U_x$ we may assume the following: $U_y$ only contains regular values, $ f^{-1}(U_y)\subset \bigcup_{x\in f^{-1}(y)}U_x$, and there exists a lift $\tilde{f}_x: \tilde{U}_x\rightarrow \tilde{U}_y$. If $z$ is another regular value of $f$ contained in $U_y$, then $z$ has $\vert\Gamma_y\vert/\vert\Gamma_z\vert$ lifts to $\tilde{U}_y$. Given that $\#\tilde{f}_x^{-1}(y)$ is a local invariant, by possibly shrinking the orbifold chart around $y$ we can assume that each one of these lifts has precisely one preimage in $\tilde{U}_x$. When projecting down from $\tilde{U}_x$ to $U_x$, some of these preimages are identified by the action of the isotropy groups $\Gamma_{x}$, so the number of points we are left with in $U_x$ is 
\[
\frac{\vert\Gamma_{y}\vert/ \vert\Gamma_z\vert}{\vert\Gamma_{x}\vert/ \vert\Gamma_{w}\vert}=\frac{\vert\Gamma_y\vert}{\vert\Gamma_{x}\vert}\cdot\frac{\vert\Gamma_{w}\vert}{\vert\Gamma_z\vert},
\]
where $w$ is a generic preimage of $z$ in $U_x$. By equivariance, all groups $\Gamma_w$ are isomorphic.
So if we compute the weighted cardinality of $z$ with the above formula in mind we see that
\begin{align*}
\#_w \cf^{-1}(z)&=\sum\limits_{w\in f^{-1}(z)} \frac{\vert\Gamma_z\vert}{\vert\Gamma_{w}\vert}=\sum\limits_{x\in f^{-1}(y)}\sum\limits_{w\in U_x\cap f^{-1}(z)}  \frac{\vert\Gamma_z\vert}{\vert\Gamma_{w}\vert}\\
  &=\sum\limits_{x\in f^{-1}(y)} \frac{\vert\Gamma_y\vert}{\vert\Gamma_{x}\vert}\cdot\frac{\vert\Gamma_{w}\vert}{\vert\Gamma_z\vert}\cdot\frac{\vert\Gamma_z\vert}{\vert\Gamma_{w}\vert}=\sum\limits_{x\in f^{-1}(y)} \frac{\vert\Gamma_y\vert}{\vert\Gamma_{x}\vert}\\
  &=\#_w \cf^{-1}(y). \hfill \qedhere
\end{align*}
\end{proof}

  \begin{definition}
Let $\cf:\cO\rightarrow \cP$ be a proper and smooth map between orbifolds of the same dimension, and let $y\in P$ be a regular value of $\cf$. We define the \emph{mod-$2$ degree} of $\cf$ at $y$ to be
$$
\deg_2(\cf;y)=\#_w\cf^{-1}(y)\mod 2.
$$
\end{definition}

\begin{proposition}
  \label{prop:locallyconstant}
The mod-$2$ degree is locally constant. 
\end{proposition}

In fact we will show below that $\deg_2(\cf;y)$ is constant in $y$ if $\mathcal P$ is connected and $\cO$ does not have a codimension-$1$ singular stratum. We start by showing homotopy invariance at regular values.

\begin{lemma}
  \label{lem:homotopyinvariance}
Let $\cf,\cg:\cO\rightarrow \cP$ be smooth and proper orbifold maps between orbifolds of the same dimension. Let $\cF:\cO\times [0,1]\rightarrow \cP$ be a proper homotopy between $\cf$ and $\cg$. Let $y\in P$ be a smooth point. Suppose that $y$ is regular for $\cF,\cf$, and $\cg$ simultaneously, and that $\Sigma^{\cO}_{\dim \cO-1}=\emptyset$. Then
\[
  \deg_2(\cf;y)=\deg_2(\cg;y). %\qquad \text{and}\qquad \deg(\cf;y)=\deg(\cg;y),
\]
\end{lemma}
\begin{proof}
  Since $y$ is a regular point of $\cf,\cg$, and $\cF$ simultaneously, Theorem~\ref{thm:preimage} tells us that $\cF^{-1}(y)$ is a full one-dimensional orbifold with boundary, while Proposition~\ref{prop:singularity} implies that all points in $ f^{-1}(y)\cup g^{-1}(y)$ are smooth, and thus the degree formulas just count the number of points in the preimage, possibly taking orientation into account. For $(x,t)\in O\times[0,1]$ we see that $\sdim(x,t)=\sdim(x)+1\geq \dim \cP=\dim \cO$. Since we assume that $\Sigma^{\cO}_{\dim \cO-1}=\emptyset$, it follows that $\Sigma^{\cO\times [0,1]}_{\dim(\cO\times [0,1])-1}=\emptyset$ and each $(x,t) \in F^{-1}(y)$ is necessarily a smooth point of $\cO\times[0,1]$. This implies that $F^{-1}(y)$ is a compact \emph{manifold} with boundary $(f^{-1}(y)\times\{0\})\cup (g^{-1}(y)\times\{1\}).$ The number of boundary components of a compact one dimensional manifold with boundary is even, which shows that $\deg_2(\cf;y)=\deg_2(\cg;y)$.
\end{proof}

\begin{lemma}
  \label{lem:invariance}
  Let $\cf:\cO\rightarrow \cP$ be a smooth and proper orbifold map between orbifolds of the same dimension. Assume that $\cP$ is connected and that $\Sigma^{\cO}_{\dim \cO-1}=\emptyset$. Let $y,z\in P$ be regular values that are smooth points. Then
  \[
  \deg_2(\cf;y)=\deg_2(\cf;z).
  \]
\end{lemma}

\begin{proof}
 Recall that a connected manifold $M$ is homogeneous, that is, for any $y,z\in M$ there exists a compactly supported isotopy $h:M\times [0,1]\rightarrow M$ such that $h_0=\id_M$ and $h_1(z)=y$, cf.~\cite[The homogeneity Lemma]{milnor1997}. For every $t\in [0,1]$ the map $h_t:M\rightarrow M$ is a diffeomorphism, so all values are regular for $h$.
  By Proposition~\ref{prop:structure}, the smooth part $M=\Sigma^\cP_{\dim \cP}$ is a connected manifold. Hence for any $y,z\in M$ there is a compactly supported isotopy $\ch:M\times [0,1]\rightarrow M$ as above. Since the isotopy is compactly supported in $M\times [0,1]$, it extends to an orbifold isotopy $\ch:\cP\times[0,1]\rightarrow \cP$ that coincides with $\cid_{\cP\times[0,1]}$ outside of a compact subset of $M\times[0,1]$. Let $\cF=\cf\circ \ch$. The orbifold map $\cF$ is a homotopy between $\cf$ and $\cg=\cf\circ \ch_1$. Then $y$ is a regular value for $\cf,\cg$ and $\cF$ simultaneously and Lemma~\ref{lem:homotopyinvariance} shows that
  \[
    \deg_2(\cf;z)=\deg_2(\cg;y)=\deg_2(\cf;y). \qedhere
  \]
\end{proof}

In the next step we combine invariance at smooth, regular points with the local invariance of the mod-$2$ degree, in order to prove invariance at all regular points.

\begin{theorem}\label{thm:independence}
  Let $\cf:\cO\rightarrow \cP$ be a smooth and proper orbifold map between orbifolds of the same dimension. Assume that $\cP$ is connected and that $\Sigma^{\cO}_{\dim \cO-1}=\emptyset$. Let $y,z\in P$ be regular points. Then
\[
  \deg_2(\cf;y)=\deg_2(\cf;z). 
  \]
\end{theorem}
\begin{proof}
  In view of Lemma~\ref{lem:invariance}, the only thing that remains to prove is that the degree is invariant for regular but not necessarily smooth points. This follows from the local invariance: suppose $y,z$ are regular but not necessarily smooth. Since the smooth stratum $\Sigma^\cP_{\dim \cP}$ is an open and dense subset of $P$, and regular values of $\cf$ are also open and dense, there exist smooth and regular values $y'$ and $z'$, sufficiently close to $y$ and $z$, respectively, such that $\deg_2(\cf;y)=\deg_2(\cf;y')$ and $\deg_2(\cf;z)=\deg_2(\cf;z')$. If we combine this with the conclusion of Lemma~ \ref{lem:invariance}, we see that
    \[
\deg_2(\cf;y)=\deg_2(\cf;y')=\deg_2(\cf;z')=\deg(\cf;z).\qedhere
    \]
  
\end{proof}
As in the manifold case, we can define an integer valued degree if the orbifolds are oriented. For every $x\in f^{-1}(y)$ choose charts $U_x$ and $U_{y}$ centered at $x$ and $y$, respectively. Let $\tilde{f}_x$ be the lift of $f$ in these charts and let $\tilde x$ be the lift of $x$. The sign $\textrm{sgn}(T_{\tilde{x}}\tilde{f}_{x})$, which tells us if the map is orientation preserving or reversing at $\tilde x$, does not depend on the choice of lift $\tilde x$.

  \begin{definition}
    Let $\cf:\cO\rightarrow \cP$ be a proper and smooth map between oriented orbifolds of the same dimension and $y\in P$ a regular value of $\cf$. Define the \emph{(integer valued or oriented) degree of $f$ at $y$} to be the oriented count of points in the preimage of $y$, namely:
\[
\deg(\cf;y)=\sum_{x\in f^{-1}(y)}\textrm{sgn}(T_{\tilde{x}}\tilde{f}_x)\cdot\frac{\vert\Gamma_y\vert}{\vert\Gamma_x\vert}.
\]
\end{definition}
The integer valued degree shares the invariance properties of the mod-$2$ degree. 
\begin{theorem}
  \label{thm:oriented}
  Let $\cf:\cO\rightarrow \cP$ be a proper and smooth map between oriented orbifolds of the same dimension. Assume that $\cP$ is connected and that $y\in P$ a regular value of $\cf$. Then the degree $\deg(\cf;y)$ does not depend on the the regular value $y$, nor on the proper homotopy class of $\cf$. 
\end{theorem}
\begin{proof}
The proof of Lemma~\ref{lemma:localinvariance} shows that also $\deg(\cf;y)$ is locally constant, since the sign $\textrm{sgn}(T_{\tilde{x}}\tilde{f}_{x})$ is locally constant in each $U_x$. Recall the proof of \ref{lem:homotopyinvariance} needs assumption $\Sigma_{\dim(\cO)-1}^{\cO}=\emptyset$, because this implies that the preimage of a regular and smooth point of the homotopy is a compact, one-dimensional manifold with boundary. But this is automatically satisfied in this case, since $\cO$ is assumed to be oriented, cf.~Proposition~\ref{prop:structure}. This implies that we can follow the proof of Lemma~\ref{lem:homotopyinvariance} to show homotopy invariance for the oriented degree. The only thing left to do is to keep track of the orientations. Since $\cO$ and $\cP$ are oriented, the compact one-dimensional manifold $F^{-1}(y)$ is also oriented, and the oriented count of the boundary components is zero. The proofs of Lemma~\ref{lem:invariance} and Theorem~\ref{thm:independence} work then ad verbatim for the oriented degree. 
\end{proof}

  Since under the assumptions of Theorems~\ref{thm:independence} and~\ref{thm:oriented} the degrees do not depend on the chosen regular value $y$, in what follows we will write $\deg_2(\cf)=\deg_2(\cf;y)$ and $\deg(\cf)=\deg(\cf;y)$, where $y$ is any regular value of $\cf$. The following statement is an immediate corollary of the fact that the degree does not depend on the regular value chosen. 
  
\begin{corollary}
Let $\cf:\cO\rightarrow \cP$ be a smooth and proper orbifold map. Assume that $\cP$ is connected and that $\Sigma^\cO_{\dim \cO-1}=\emptyset$. Then if either $\deg_2(\cf)\not=0$ or $\deg(\cf)\not=0$ then the underlying map $f$ is surjective. Here we assume that $\cO$ and $\cP$ are oriented for the integer valued degree to be defined. 
\end{corollary}

Another consequence of the fact that the degree of an orbifold map is independent of the choice of regular value is the next proposition, which shows that the degree behaves multiplicatively with respect to composition. 

  \begin{proposition}
    \label{prop:multiplicative}
    Let $\cf:\cO\rightarrow \cP$ and $\cg:\cP\rightarrow \cS$ be smooth and proper orbifolds maps between orbifolds of the same dimension. Assume $\cP$ and $\cS$ to be connected and that $\Sigma^\cO_{\dim \cO-1}=\Sigma^\cP_{\dim \cP-1}=\emptyset$. Then for any regular value $y$ of $\cg\cf$ we have
    \[
\deg_2(\cg\cf)=\deg_2(\cg)\deg_2(\cf)\qquad\text{and}\qquad      \deg(\cg\cf)=\deg(\cg)\deg(\cf),
\]
where we assume that the orbifolds are oriented for the integer degree to be well-defined. 

  \end{proposition}

\begin{proof}
We prove this for the integer degree, the case of the mod-$2$ degree is similar. Let $z$ be a regular value of the composition $\cg \cf$. Then $z$ must be a regular value of $\cg$, and all $y\in g^{-1}(z)$ must be regular values for $\cf$ and we get
\begin{align*}
\deg(\cg \cf) & =\sum_{x\in (g f)^{-1}(z)}\textrm{sgn}(T_{\tilde{x}}(\widetilde{g f})_x)\cdot\frac{\vert\Gamma_z\vert}{\vert\Gamma_x\vert}\\
& =\sum_{y\in g^{-1}(z)}\sum_{x\in f^{-1}(y)}\textrm{sgn} (T_{\tilde{y}}\tilde{g}_y\,T_{\tilde{x}}\tilde{f}_x)\cdot \frac{\vert\Gamma_z\vert}{\vert\Gamma_y\vert}\cdot\frac{\vert\Gamma_y\vert}{\vert\Gamma_x\vert}\\
                    & =\sum_{y\in g^{-1}(z)}\textrm{sgn}( T_{\tilde{y}}\tilde{g}_y)\cdot \frac{\vert\Gamma_z\vert}{\vert\Gamma_y\vert}\,\sum_{x\in f^{-1}(y)}\textrm{sgn}( T_{\tilde{x}}\tilde{f}_x)\cdot \frac{\vert\Gamma_y\vert}{\vert\Gamma_x\vert}\\
  & =\sum_{y\in g^{-1}(z)}\textrm{sgn}(T_{\tilde{y}}\tilde{g}_y)\cdot \frac{\vert\Gamma_z\vert}{\vert\Gamma_y\vert}\cdot\deg(\cf;y)\\
& =\deg(\cg) \deg(\cf).
\end{align*}
Notice that we have used the fact that the degree of $\cf$ is independent of the regular value $y$.
\end{proof}

\begin{theorem}
  Let $\cf,\cg:\cO\rightarrow \cP$ be two proper orbifold maps whose underlying maps $f,g:O\rightarrow P$ are the same. Assume that $\cP$ is connected and that $\Sigma^{\cO}_{\dim \cO-1}=\emptyset$. Then
  $$
\deg_2(\cf)=\deg_2(\cg)\qquad \text{and}\qquad \deg(\cf)=\deg(\cg),
$$
where we assume the orbifolds are oriented for the integer degree to be well-defined. 
\end{theorem}
\begin{proof}
 Since regular values of both $\cf$ and $\cg$ are open and dense in $\cP$, there is a common regular value. If we compute the mod-$2$ degree at this point, we see that the parameters involved are just the number of preimages and the order of the isotropy groups of these preimages, which are all independent of the chosen lifts. In the oriented degree case, the additional sign term appearing in the degree formula is also determined by the underlying map, which can be most easily seen at smooth points.
\end{proof}

\section{Quotient orbifolds}

\label{sec:quotient}
Orbifolds arise in a natural way if we take the quotient of a smooth manifold $M$ by the effective and locally free action of a compact Lie group $G$. Weighted projective spaces fall into this category, as quotients of a circle action on an odd dimensional sphere. For a quotient orbifold $\mathcal{O}=M//G$, the underlying space is just the topological quotient $M/G$ and orbifold charts can be constructed as follows: the Slice Theorem provides for each $x\in M$ a submanifold $S_x\cong \mathbb{R}^n$ of $M$ containing $x$ and which is invariant under the action of the stabilizer $G_x=\textrm{Stab}(x)$ (here $n=\dim M-\dim G$). Such a ``slice" is the image, under the exponential map associated to an auxiliary Riemannian metric, of a fibre of an $\epsilon$-tubular neighborhood of the orbit of $x$. An orbifold chart around $[x]$ is given by $(S_x,G_x, \phi_x:S_x\rightarrow S_x/G_x).$ 

We now show that equivariant mappings induce orbifold maps between the quotients. 

  \begin{proposition}
    \label{prop:slicemap}
Suppose $G_1$ and $G_2$ are compact Lie groups that act smoothly, effectively and locally freely on the smooth manifolds $M_1$ and $M_2$. Let $\Theta:G_1\rightarrow G_2$ be a Lie group homomorphism. Let $\hat{f}:\,M_1\rightarrow M_2$ be a $\Theta$-equivariant smooth map, i.e. $\hat f(\gamma  x)=\Theta(\gamma )\hat f(x)$ for all $\gamma \in G_1$ and $x\in M_1$. Then $\hat{f}$ induces a complete, smooth orbifold map $\cf:\, M_1//G_1\rightarrow M_2//G_2$.
\end{proposition}

\begin{proof}
  The map $\hat{f}$ maps $G_1$-orbits to $G_2$-orbits and therefore the map $f([x])=[\hat{f}(x)]$ is a well-defined continuous map between the quotient $M_1/G_1$ and $M_2/G_2$. If $\gamma \in G_x$, then $\Theta(\gamma )\cdot \hat{f}(x)=\hat{f}(\gamma \cdot x)=\hat{f}(x)$, so $\Theta(\gamma )\in G_{\hat{f}(x)}$. Thus for every $x\in M_1$, the group homomorphism $\Theta$ restricts to a group homomorphism $\Theta_x:\,G_x\rightarrow G_{\hat{f}(x)}$. Let $(S_x, G_x, \phi_x)$ be an orbifold chart around $[x]$ and recall that $S_x$ is a $G_x$-invariant submanifold of $M_1$ such that  $x\in S_x$. Similarly, choose a chart $(S_{\hat{f}(x)},G_{\hat{f}(x)},\phi_{\hat{f}(x)})$ around $[\hat{f}(x)]$. Since $\hat f$ is neither immersive nor submersive at $x$, there is no apriori reason for $\hat f$ to map the slice $S_x$ to the slice $S_{\hat f(x)}$. We now construct a map $\tilde{f}_{[x]}$ that does map $S_x$ to $S_{\hat f(x)}$.
  
  Recall that the subgroup $G_{\hat f(x)}\subset G_2$ is finite. Therefore there exists a neighborhood $U\subset G_2$ of the identity such that $U\cap G_{f(x)}=\{e\}$, the trivial subgroup. By the definition of a slice, the set $V=\{\gamma z\,|\, \gamma \in U\quad z\in S_{\hat f(x)}\}$ is a neighborhood of $\hat f(x)$. Since $\hat f$ is continuous, we can shrink $S_x$ so that $\hat f(S_x)\subset V$. For each $y\in S_x$, the orbit through $\hat f(y)$ intersects the slice $S_{\hat f(x)}$ in a point and there exists a unique $k(y)\in U$ such that $k(y)\hat f(y)\in S_{\hat f(x)}$ (recall that $\Gamma_x\cap U$ is the identity, cf.~Figure~\ref{fig:slicemap}). Now let $\gamma \in G_x$ and $y\in S_x$. Then 
\[
k(y)\Theta_x(\gamma )^{-1}\hat f(\gamma y)=k(y)\hat f(\gamma ^{-1}\gamma y)=k(y)f(y)\in S_{\hat f(x)}.
\]
Since $\Theta_x(\gamma )\in G_{\hat f(x)}$, it follows that also $\Theta_x(\gamma )k(y)\Theta_x(\gamma )^{-1}\hat f(\gamma y)\in S_{\hat f(x)}$. Moreover, since $\Theta_x(\gamma )k(y)\Theta_x(\gamma )^{-1}\in U$ for $y$ sufficiently close to $x$, we must have that $k(\gamma y)=\Theta_x(\gamma )k(y)\Theta_x(\gamma )^{-1}$.

Define the map $\tilde f_{[x]}:S_x\rightarrow S_{\hat{f}(x)}$ via $\tilde f_{[x]}(y)=k(y)\hat f(y)$. We need to check that this map is $\Theta_x$-equivariant: let $\gamma \in G_x$ and $y\in S_x$, then by the transformation rule for $k$ and the equivariance of $\hat f$ we have
\begin{align*}
  \tilde f_{[x]}(\gamma  y)&=k(\gamma y)\hat f(\gamma  y)\\
               &=\Theta_x(\gamma )\Theta_x(\gamma )^{-1}k(\gamma y)\Theta_x(\gamma ) \hat f(y)\\
  &=\Theta_x(\gamma ) k(y)\hat f(y)=\Theta_x(\gamma )\tilde f_{[x]}(y).
\end{align*}
Hence we have constructed a $\Theta_x$-equivariant lift $\tilde f_{[x]}$ of $f$ in the chart $(S_x, G_x, \phi_x)$ around each point $[x]\in M_1//G_1$, and thus defined a smooth orbifold map $\cf$. 
\end{proof}
  
    \begin{figure}
\def\svgwidth{.8\textwidth}
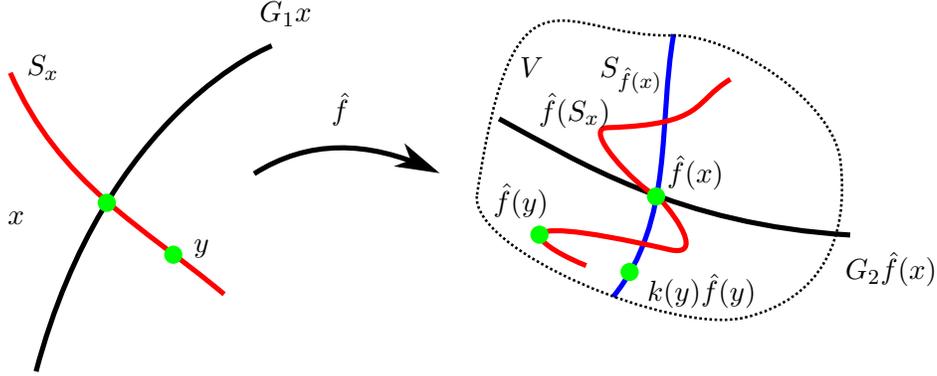
\caption{This figure illustrates the construction in Proposition~\ref{prop:slicemap}. A complete orbifold map $\cf$ is constructed from an equivariant map $\hat f$. The red slice $S_x$ through $x$ is not mapped to the blue slice $S_x$. However, for every element $y\in S_x$ there exists a unique element $k(y)\in G_2$, close to the identity, such that $k(y)\in S_{\hat f(x)}$.}
  \label{fig:slicemap}
\end{figure}

\begin{remark}
  \label{remark:proper}
With a small additional effort, the result can also be proved if the groups $G_1,G_2$ are not necessarily compact. It is crucial to demand that the actions are proper, i.e. the maps $\Phi_i:G_i\times M_i\rightarrow G_i\times M_i$ defined by $\Phi_i(\gamma ,x)=(\gamma x,x)$ are proper, so that slices exist. 
     \end{remark}

     \begin{remark}
The following result was already proved by Satake in \cite{S1956}: if $\mathcal{O}$ is an effective orbifold, then it can be presented as $M//G$, where $M=\textrm{Fr}_{O(n)}\mathcal{O}$ denotes the orthonormal frame bundle of some Riemannian metric and $G=O(n)$.
The orthonormal frame bundle $\textrm{Fr}_{O(n)}\mathcal{O}$ is an orbifold fibre bundle over $\mathcal{O}$. It is constructed as follows (cf. \cite{adem2007}): given an orbifold chart $\{(\tilde{U},\Gamma,\phi)\}$, one considers the manifold $\textrm{Fr}_{O(n)}(\tilde{U})$, consisting of pairs $(\tilde{x},F)$, where $\tilde{x}$ is a point in $\tilde{U}$ and $F$ is an orthonormal frame at $\tilde{x}$. The group $\Gamma$ acts on these pairs by $\gamma\cdot (\tilde{x},F)=(\gamma\cdot\tilde{x}, T \gamma_{\tilde{x}}F).$ This action of $\Gamma$ on frames is free, so the local model, the quotient $\textrm{Fr}_{O(n)}(\tilde{U})/\Gamma$, is smooth. 
         The space $\textrm{Fr}_{O(n)}(\mathcal{O})$ is obtained by gluing together these local charts using suitable transition functions. The action of the orthogonal group on $\textrm{Fr}_{O(n)}(\tilde{U})$ induces a $O(n)$-action given, in local charts, by $[(\tilde{x},F)]\cdot A=[(\tilde{x},FA)]$, which is compatible with the gluing and therefore extends to a global action on $\textrm{Fr}_{O(n)}(\mathcal{O})$. The (orbifold) quotient  $\textrm{Fr}_{O(n)}(\mathcal{O})//O(n)$ is isomorphic to the original orbifold $\mathcal{O}$.
         
         Proposition~\ref{prop:slicemap} in particular shows that equivariant maps of frame bundles induce complete smooth orbifold maps. However the converse does not immediately follow. The problem is that the lifted tangent map $T_{\tilde x}\tilde f_x$ does not define a map of frames, unless $T_{\tilde x}\tilde f_x$ is invertible. We are interested in the question if all complete orbifold mappings can be canonically lifted to an equivariant map.
\end{remark}

\section{Examples}

\label{sec:examples}

\begin{figure}
\def\svgwidth{.3\textwidth}
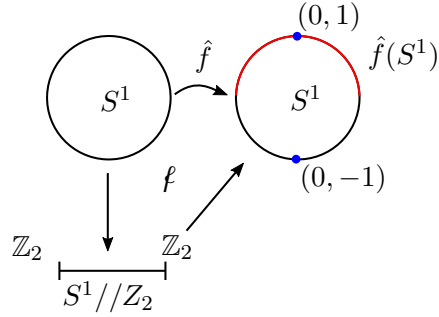
\caption{In Example~\ref{ex:circle1} we discuss an orbifold map whose mod-$2$ degree does depend on the chosen value. The image of $\hat{f}$, which equals the image of the underlying map $f$ of $\cf$, is the top half of the circle $S^1$. It is shown in red in the Figure. We have that $\deg_2(\hat f;y)=0$ is independent of the regular value $y$, while the degree of $\cf$ depends on the value chosen, namely: $\deg_2(\cf;(0,1))=1$ and $\deg_2(\cf;(0,-1))=0$. } 
  \label{fig:circle1}
\end{figure}

\begin{example}[An orbifold map whose degree depends on the choice of regular value]
\label{ex:circle1}
  The mod-$2$ degree of a proper map $\cf:\cO\rightarrow \cP$ is independent of the regular value $y\in P$ if $\cP$ is connected and $\cO$ does not have a codimension $1$ singularity. A hypothesis like this is necessary as the following example illustrates:

Let $S^1\subset \mR^2$ be the standard circle. Then $\mZ_2$ acts on it by reflection in the second coordinate. The underlying space of the orbifold $S^1//\mZ_2$ is an orbifold is a closed interval, whose boundary points have $\mZ_2$ isotropy. The map $\hat f:S^1\rightarrow S^1$ given by $\hat f(x,y)=\frac{1}{\sqrt{x^2+y^4}}(x,y^2)$ satisfies $\hat f(x,y)=\hat f(x,-y)$, hence is equivariant if the domain is equipped with the action of $\mZ_2$ above and the codomain has a trivial group acting. By Proposition~\ref{prop:slicemap} the map $\hat f$ descends to a complete orbifold map $\cf:S^1//\mZ_2\rightarrow S^1$ (here $\Theta$ is the trivial homomorphism). Both $(0,-1)$ and $(0,1)$ are regular values for this map. One easily checks that $f^{-1}((0,1))$ contains one preimage that is a smooth point. Hence $\deg_2(\cf;(0,1))=1$. But $\deg_2(\cf;(0,-1))=0$ as this does not lie in the image of $\cf$. This map is depicted in Figure~\ref{fig:circle1}.
\end{example}

  \begin{example}[A contractible map with non-zero degree at all regular values]
  \label{ex:circle2}
    Consider once more the action of $\mZ_2$ on $S^1$ by reflection in the second coordinate, and the following two maps from $S^1$ to itself:
\[
\hat f(x,y)=\frac{(x,e^{-1/y^2})}{x^2+e^{-2/y^2}},\ \textrm{and}\ \hat g(x,y)=\frac{(x,\textrm{sign}(y)e^{-1/y^2})}{x^2+e^{-2/y^2}} \ \textrm{if}\ y
\neq 0,\quad \hat{f}(x,0)=\hat{g}(x,0)=(x,0).
\]
Let $\Theta^f:\mZ_2\rightarrow \mZ_2$ be the trivial homomorphism and $\Theta^g:\mZ_2\rightarrow \mZ_2$ be the identity. Then $\hat f$ and $\hat g$ are $\Theta^f$ and $\Theta^g$ equivariant, respectively, and by Proposition~\ref{prop:slicemap} they induce orbifold maps $\cf,\cg: S^1//\mZ_2\rightarrow S^1//\mZ_2$. The underlying maps $f$ and $g$ are equal. All smooth points are regular and $\deg_2(\cf;z)=\deg_2(\cg;z)$ for all smooth points. Notice that $\cg$ is not contractible, while $\cf$ is. In fact, $\cg$ it is homotopic to the identity. 
\end{example}
\begin{figure}
\def\svgwidth{.4\textwidth}
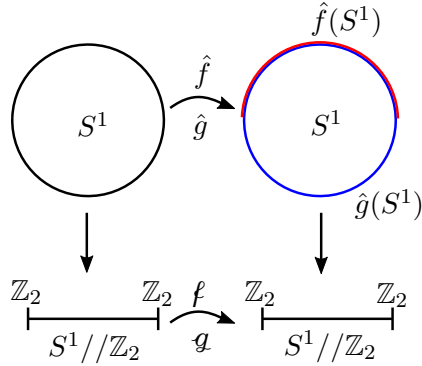
\caption{This figure depicts the maps constructed in Example~\ref{ex:circle2}. The maps $\hat f$ and $\hat g$ are $\Theta^f$ and $\Theta^g$ equivariant, respectively, and descend to orbifold maps $\cf$ and $\cg$. The underlying maps $f$ and $g$ are equal and $\deg_2(\cf;y)=\deg_2(\cg;y)$ for all regular values, that is, all the smooth points of $S^1//\mZ_2$. The orbifold map $\cf$ is contractible while $\cg$ is not. } 
  \label{fig:circle2}
\end{figure}

\subsection{Degrees of maps between weighted projective spaces}

Let $q=(q_0,\ldots,q_n)$ be an ($n+1$)-tuple of positive integers. The group $\mC^*=\mC\setminus\{0\}$ acts on $\mC^{n+1}\setminus\{0\}$ as follows: $\gamma\cdot(z_0,\ldots,z_n)=(\gamma^{q_0}z_0,\ldots, \gamma^{q_n}z_n).$ The weighted projective space $\mC \mP^n(q)$ is the quotient of $\mC^{n+1}\setminus\{0\}$ by this action. We denote the equivalence class of $(z_0,\ldots,z_n)\in \mC^{n+1}\setminus\{0\}$ by $[z_0:\ldots:z_n]_q$. If $q_0=q_1=\ldots=q_n=1$, then we recover the ordinary projective space. The action is proper and we will assume that the weights $q_i$ are coprime, so that the action is effective. The weighted projective space $\mC\mP^n(q)$ has then the structure of an effective orbifold. The maps $\hat f_q:\mC^{n+1}\setminus\{0\}\rightarrow \mC^{n+1}\setminus\{0\}$ and $\hat g_q:\mC^{n+1}\setminus\{0\}\rightarrow \mC^{n+1}\setminus\{0\}$ given by
\[
\hat f_q(z_0,\ldots,z_n)=(z_0^{q_0},\ldots, z_n^{q_n})\quad \textrm{and}\quad \hat g_q(z_0,\ldots,z_n)=(z_0^{\frac{\textrm{lcm}(q)}{q_0}},\ldots, z_n^{\frac{\textrm{lcm}(q)}{q_n}})
\]
are $\id_{\mC^*}$ equivariant, and so they induce orbifold maps $\cf_q: \mC\mP^n\rightarrow \mC\mP^n(q)$ and $\cg_q:\mC\mP^n(q)\rightarrow \mC\mP^n$. The underlying maps satisfy
\[
f_q([z_0:\ldots:z_n])=[z_0^{q_0}:\ldots: z_n^{q_n}]_q\quad \textrm{and}\quad g_q([z_0:\ldots:z_n]_q)=[z_0^{\frac{\textrm{lcm}(q)}{q_0}}:\ldots: z_n^{\frac{\textrm{lcm}(q)}{q_n}}].
\]
We invoke~Remark~\ref{remark:proper} to construct these maps, but notice that is also possible to normalise the coordinates and define $\hat f_q$ and $\hat g_r$ as maps from $S^{2n+1}$ onto $S^{2n+1}$, with the compact Lie group $S^1$ acting on $S^{2n+1}$ via the same formula, and apply Proposition~\ref{prop:slicemap} directly.
Here $\textrm{lcm}(q)$ denotes the least common multiple of the $q_i$'s, so in particular, when the $q_i$'s are pairwise coprime, $\textrm{lcm}(q)=q_0\cdot\ldots\cdot q_n$. The map $\cf_q$ has degree $q_0\cdot \ldots \cdot q_n$: to see this, consider the orbifold chart $(\tilde U_x,\Gamma_x,\phi_x)$ around % $(U_0,\Gamma_0)$ around
  $x=[1:0:\ldots :0]_q$, where 
\[
U_x=\{[z_0:\dots :z_n]_q\in \mC\mP^n(q)\,:\, z_0\neq 0\}
\]
and $\Gamma_x=\mZ_{q_0}$ acts on $\tilde U_x=\mC^n$ by $\xi\cdot (w_1,\ldots,w_n)=(\xi^{q_1}w_1,\ldots,\xi^{q_n}w_n)$. The map $\phi_x: \tilde U_x/\Gamma_x\rightarrow U_x$ is defined by $\phi_x([w_1,\ldots,w_n])=[1:w_1:\ldots :w_n]_q$. The point $[1:1:\ldots :1]\in \mC\mP^n$ is mapped by $\cf_q$ to the point $[1:1:\ldots :1]_q$. Notice that this is a smooth point as all the $q_i$'s are coprime. By now choosing the standard homogeneous chart defined by $z_0\neq0$ as a (manifold) chart around $[1:1:\ldots :1]$ in $\mathbb{CP}^n$, and $\tilde U_{x}$ as a chart around $[1:1:\ldots :1]_q$, we see that $\cf_q$ lifts, in these charts, to the map $\left(\tilde{f}_q\right)_x(w_1,\ldots,w_n)=(w_1^{q_1},\ldots,w_n^{q_n})$. Hence $[1:1:\ldots :1]_q$ is a regular value and $\mathrm{sgn}\left(d_{\tilde x}\left(\tilde f_q\right)_x\right)=1$. Moreover, under $\cf_q$ it has exactly $q_0\cdot \ldots\cdot q_n$ preimages (namely all the elements of $\mC\mP^n$ of the form $[\xi_0:\ldots:\xi_n]$ with $\xi_i$ a $q_i$-th root of unity). Thus
\[
\deg(\cf_q)=q_0\cdot \ldots \cdot q_n.
\]

The claim about the degree of the second map, $\cg_q$, can be verified as follows. The composition $\cg_q\circ \cf_q$ maps $[z_0:\ldots : z_n]$  to $[z_0^{\textrm{lcm}(q)}:\ldots :z_n^{\textrm{lcm}(q)}]$ and it is a degree $\textrm{lcm}(q)^n$ self-map of $\mC\mP^n$. In view of the multiplicativity of the degree (Proposition~\ref{prop:multiplicative}), we can conclude that
\[
  \deg(\cg_q)=\frac{\textrm{lcm}(q)^n}{q_0\cdot\ldots\cdot q_n}
\]

We can also introduce a second weighted projective space $\mC \mP^n(r)$, with different weights, and consider the composition $\ch_{rq}=\cf_r\circ\cg_q:\,\mC\mP^n(q)\rightarrow \mC\mP^n(r)$, whose underlying map satisfies
\[
  h_{rq}([z_0:\ldots:z_n]_q)=[z_0^{\frac{\textrm{lcm}(q)}{q_0}\cdot r_0}:\ldots: z_n^{\frac{\textrm{lcm}(q)}{q_n}\cdot r_n}]_r.
\]
Again by the multiplicativity of the degree, Proposition~\ref{prop:multiplicative}, we see that
  \[
    \deg(\ch_{rq})=\frac{\textrm{lcm}(q)^n}{q_0\cdot\ldots\cdot q_n}\cdot r_0\cdot \ldots\cdot  r_n.
  \]

We would like to use this example to underline the fact that regular values of an orbifold map are not necessarily smooth points, and conversely. 
For instance, let $q=(1,\dots,1,k)$ and consider the map $\cf_q$ defined above. It lifts to the identity in a neighbourhood of $[0:\ldots :0:1]_q$. 
The point $[0:\dots :0:1]_q\in \mC \mP^n(q)$ is thus a non-smooth point (it has nontrivial isotropy, namely $\mZ_k$), but it is a regular value of the mapping $\cf_q$.
Its preimage consists of only one point, namely $[0:\ldots :0:1]\in \mC\mP^n$, and by the weighted count in our definition it follows that the degree of $\cf_q$ at this point is $k$. 
On the other hand, the point $[1:0:\ldots :0]_q$ is a smooth point of $\mC \mP^n(q)$, but it is not a regular value of $\cf_q$, since in standard orbifold charts it lifts to $(w_1,\dots,w_n)\mapsto (w_1,\ldots, w_n^k)$ in a neighbourhood of the origin of $\mC^n$.

\begin{example}[The degree of an orbifold covering.]

  Let $G$ be a finite group acting effectively by orientation preserving diffeomorphisms on a connected and oriented manifold $M$. Then we have the oriented orbifold $\cO:=M//G$. The identity $\hat p:M\rightarrow M$ is an equivariant map, where on the domain we consider the trivial $G$-action on $M$, while on the codomain we let $G$ act in the prescribed way. 
This induces the orbifold covering map $\mathcal{p}:M\rightarrow \cO$. All values of $\mathcal{p}$ are regular, since the lifts to the orbifold charts are all identity maps. This also implies that the signs that we encounter in the formula below are all equal to one. For each $y\in P$, there are $\frac{|G|}{|\Gamma_y|}$ preimages of $y$ under $\mathcal{p}$, and the degree of the orbifold covering map is
  $$
  \deg(\mathcal p;y)=\sum_{x\in \mathcal{p}^{-1}(y)} \textrm{sgn}(T_{\tilde{x}}\hat p_x)\cdot \frac{\vert\Gamma_y\vert}{\vert\Gamma_x\vert}=\sum_{x\in \mathcal{p}^{-1}(y)}\vert\Gamma_y\vert=\frac{\vert G\vert}{\vert \Gamma_y\vert}\cdot\vert \Gamma_y\vert
  =|G|.$$
  Now consider a pair of finite groups $G_1,G_2$ acting by orientation-preserving diffeomorphisms on connected and oriented manifolds $M_1$ and $M_2$. Let $\Theta:G_1\rightarrow G_2$ be a homomorphism and $\hat f:M_1\rightarrow M_2$ a proper and $\Theta$-equivariant map. Then we obtain the commutative diagram
  $$
  \xymatrix{M_1\ar[r]^{\hat{f}}\ar[d]_{\mathcal{p_1}}&M_2\ar[d]^{\mathcal{p_2}}\\
    \cO_1 \ar[r]_{\cf}& \cO_2}
  $$
  which together with Proposition~\ref{prop:multiplicative} (multiplicativity of the degree) and our previous computation of                     $\deg(p_i)=\vert G_i\vert$, implies that the degrees of $\hat f$ and $\cf$ are related by
  $$
  \deg(\cf)=\deg(\hat f)\cdot\frac{\vert G_2\vert}{\vert G_1\vert}.
  $$
\end{example}
  
\bibliographystyle{abbrv} 
 \bibliography{orbifolds}

\end{document}